\documentclass[11pt]{aptpub}

\usepackage[utf8x]{inputenc}
\usepackage[a4paper]{geometry}

\authornames{F. PORTIER, J. SEGERS} 
\shorttitle{Monte Carlo integration} 

\usepackage{microtype}
\RequirePackage[numbers]{natbib}
\usepackage{enumerate}
\usepackage{dsfont}
\usepackage{url}

\newcommand{\Var}{\operatorname{var}}

\newcommand{\trace}{\operatorname{tr}}
\newcommand{\diff}{\mathrm{d}}
\newcommand{\expec}{\mathbb{E}}
\newcommand{\reals}{\mathbb{R}}
\newcommand{\pr}{\mathbb{P}}
\newcommand{\eps}{\varepsilon}
\newcommand{\oh}{\mathrm{o}}
\newcommand{\Oh}{\mathrm{O}}
\newcommand{\dto}{\rightsquigarrow}
\newcommand{\argmin}{\operatornamewithlimits{\arg\min}}
\newcommand{\1}{\mathds{1}}

\usepackage{color}
\definecolor{chocolate(traditional)}{rgb}{0.48, 0.25, 0.0}
\definecolor{darkraspberry}{rgb}{0.53, 0.15, 0.34}
\definecolor{darkmagenta}{rgb}{0.55, 0.0, 0.55}

\begin{document}

\title{Monte Carlo integration with a growing number of control variates} 

\authorone[T\'el\'ecom Paris]{Fran\c{c}ois Portier} 
\addressone{LTCI, T\'el\'ecom Paris, Institut polytechnique de Paris, rue Barrault, 75013 Paris, France. Email: francois.portier@gmail.com} 
\authortwo[UCLouvain]{Johan Segers}
\addresstwo{Institut de statistique, biostatistique et sciences actuarialles, LIDAM, UCLouvain, Voie du Roman Pays 20, B-1348 Louvain-la-Neuve, Belgium. Email: johan.segers@uclouvain.be}

\begin{abstract}
 It is well known that Monte Carlo integration with variance reduction by means of control variates can be implemented by the ordinary least squares estimator for the intercept in a multiple linear regression model. A central limit theorem is established for the integration error if the number of control variates tends to infinity. The integration error is scaled by the standard deviation of the error term in the regression model. If the linear span of the control variates is dense in a function space that contains the integrand, the integration error tends to zero at a rate which is faster than the square root of the number of Monte Carlo replicates. Depending on the situation, increasing the number of control variates may or may not be computationally more efficient than increasing the Monte Carlo sample size.
\end{abstract}



\section{Introduction}

Numerical integration algorithms can generally be characterized by (a) the \textit{integration points} at which the integrand is evaluated and (b) the \textit{integration weights} describing how to combine the evaluations of the integrand.  Popular algorithms include the Riemann sums method, the Gaussian quadrature rule, and the classical Monte Carlo method.
Those algorithms are usually compared by looking at the integration error for a given number, say $n \geq 1$, of integration points. Two types of methods can be distinguished. The ones that are based on deterministic integration points (e.g., equally spaced points) including the Riemann sums and the Gaussian quadrature, and the ones that generate randomly the integration points including the Monte Carlo method. The deterministic methods reach an accuracy of order $n^{-k/d}$ \citep[Theorem~1]{novak:2016}, where $k$ stands for the regularity of the integrand and $d$ is the dimension of the integration domain, whereas random methods are subjected to an optimal error bound of order $n^{-k/d}n^{-1/2}$ \citep[Theorem~3]{novak:2016}. For instance, the naive Monte Carlo method, which does not use any regularity of the integrand, converges at the rate $n^{-1/2}$. Those error bounds testify to the benefits of random methods over deterministic ones especially when facing high-dimensional settings. 

The method of control variates is a popular technique in Monte Carlo integration that aims at reducing the variance of the naive Monte Carlo estimate by taking advantage of the regularity of the integrand \citep{Glasserman2003,glynn+s:2002,mcbook,casella+r:2003}. It is based on the introduction of auxiliary functions, called control variates, with known integral. Given a fixed number of control variates, the method consists in (i)~fitting a linear combination of the control variates to the integrand, and (ii)~using the fitted function in a modified Monte Carlo procedure. As noted in \cite{oates:2016}, the fit of the integrand in step~(i) generally uses the integration points, which makes the control variate method a post-hoc scheme, i.e., that might be done after sampling the integration points. The control variate approach is quite general as it allows to recover several famous examples from numerical integration. In dimension~$1$, the Newton--Cotes rule with random interpolation points can be recovered by taking the polynomials of degree smaller than $n-1$ as control variates. The post-stratification method can also be recovered by combining the indicators of a given partition of the integration domain (see Example~\ref{ex:indicator} in Section~\ref{sec:examples} below). When no control variates are used, it coincides with the classical Monte Carlo method.

An important field of application of the control variates method is financial engineering where it has been used for Asian option pricing in the Black--Scholes model \cite[Example~4.1.2]{Glasserman2003} or to solve backward stochastic differential equations \cite{gobet+l:2010}. More recently, it has been helpful in reinforcement learning to accelerate the estimation of the optimal policy \cite{jie+a:2010}. As highlighted in the present paper the method is efficient when many integrals need to be computed. This is the case for instance in quantile estimation \citep{HesterbergNelson1998}, option pricing \citep{glasserman+y:2005}, and likelihood computation in statistical models with latent variables \citep{portier+s:2018v3}, which arise frequently in economics \cite{mcfadden:2001} and medicine \citep[Examples~4, 6 and~9]{mcculloch+s:2001}. Finally, note that using importance sampling permits to recover the Lebesgue measure as the reference measure \citep[Theorem~2]{owen+z:2000} which in turn allows the use of many control variates such as polynomials, indcators, splines and Gaussian mixtures.

As illustrated by the well-known Runge phenomenon in approximation theory, enlarging the number of control variates does not necessarily improve the method. A key question then, which will be central in the paper, is related to the number of control variates that should be used in the procedure. The possibility of letting the number of control functions tend to infinity is already alluded to in \citep[Theorem~3]{glynn+s:2002}, who show that, for control functions arising as the power sequence of a given function, the variance of the limiting normal distribution of the error of the control variate method converges to the variance of the residual of the conditional expectation of the integrand given the initial control function. However, this result is still cast within the setting of a fixed number of control variates, i.e., the number of control variates does not depend on the Monte Carlo sample size. A recent proposal in \citep{oates:2016} is to construct the linear fit to the integrand in step~(i) above as an element of a reproducing kernel Hilbert space, whose dimension grows with the sample size $n$. Their approach leads to a convergence rate that is at least as fast as $n^{-7/12}$ and thus improves over the Monte Carlo rate. Further refinements are given in \citep{oates+c+b+g:2017}, with tighter error bounds depending on the smoothness of the integrand.

In this paper, we adopt the original control variate framework but allow the number of control variates $m = m_n$ to grow with $n$. Among the six control variate estimators in \cite{glynn+s:2002}, only one possesses the property of integrating the constants and the control functions without error. This is the one we promote and study in this paper. We use the denomination ordinary least squares Monte Carlo (OLSMC) because of the well-known link \citep{Glasserman2003} with the ordinary least squares estimator for the intercept in a multiple linear regression model with the integrand as dependent variable and the control variates as explanatory variables.

Our main result is that when $m_n \to \infty$ but $m_n = \oh(n^{1/2})$ and under reasonable conditions on the control functions and the integrand, the OLSMC estimator obeys a central limit theorem with the non-standard rate $n^{-1/2} \sigma_n$, where $\sigma_n$ is the standard deviation of the error term in the aforementioned multiple regression model. Moreover, we show that the common estimator $\hat{\sigma}_n$ of the standard deviation defined via the residual sum of squares is consistent in the sense that $\hat{\sigma}_{n} / \sigma_{n} \to 1$ in probability. This fact guarantees the asymptotic coverage of the usual confidence intervals.

If $\sigma_n \to 0$, then the convergence rate of the OLSMC is faster than the $n^{-1/2}$ rate of the ordinary Monte Carlo procedure. Still, this acceleration is offset by an increased computational cost, from $\Oh(n)$ operations for ordinary Monte Carlo to $\Oh(nm_n^2)$ for the control variate method, a number which can be brought down to $\Oh(nm_n)$ in certain situations. A more balanced comparison arises when we allow the naive Monte Carlo method to compete on the basis of a larger sample size, matching computation times. Whether or not the investment in $m_n$ control variates is worth the effort then depends on the exact speed at which $\sigma_n$ tends to zero, as is illustrated by examples. 

In Section~\ref{sec:method}, we recall the method of control variates, highlighting a formulation in terms of projections which is useful later on. A central limit theorem when the number of control variates tends to infinity is developed in Section~\ref{sec:asym}. Its formulation allows for a sequence of integrands and for a triangular array of control variates. The balance between accelerated convergence rate and increased computational cost is investigated in Section~\ref{sec:cost}. Examples of families of control functions are presented in Section~\ref{sec:examples} while some concluding comments are given in Section~\ref{sec:discussion}. All proofs are relegated to Section~\ref{sec:proofs}.

\section{Control variates and orthogonal projections}
\label{sec:method}

\subsection{Control variates}

Let $(S, \mathcal{S}, P)$ be a probability space and let $f \in L^2(P)$ be a real function on $S$ of which we would like to calculate the integral $\mu = P( f ) = \int_S f(x) \, P(\diff x) $. Let $X_1, \ldots, X_n$ be an independent random sample from $P$ on a probability space $(\Omega, \mathcal{A}, \pr)$ and let $P_n$ be its empirical distribution. The Monte Carlo estimate of $\mu$ is $\hat{\mu}_n = P_n( f ) = n^{-1} \sum_{i=1}^n f(X_i)$. The Monte Carlo estimator is unbiased and has variance $\Var( \hat{\mu}_n ) = n^{-1} \sigma^2(f)$, where $\sigma^2(f) = P [ \{ f - P(f) \}^2 ]$. By the central limit theorem, $\sqrt{n} ( \hat{\mu}_n - \mu ) \dto \mathcal{N}(0, \sigma^2(f))$ as $n \to \infty$, where the arrow $\dto$ denotes convergence in distribution.

The use of control variates is one of many methods to reduce the asymptotic variance of the Monte Carlo estimator. Let $h_1, \ldots, h_m \in L^2(P)$ be functions with known expectations. Without loss of generality, assume that $P(h_j) = 0$ for all $j = 1, \ldots, m$. For every column vector $\beta \in \reals^m$, we obviously have $\mu = P( f - \beta' h )$, where $h = (h_1, \ldots, h_m)'$ is the column vector with the $m$ control functions as elements. But then $\hat{\mu}_n(\beta) = P_n ( f - \beta' h )$ is an unbiased estimator of $\mu$ too, with variance $\Var\{ \hat{\mu}_n(\beta) \} = n^{-1} \sigma^2( f - \beta' h )$.

The asymptotic variance $\sigma^2( f - \beta' h ) $ is minimal if $\beta$ is equal to
\begin{equation}
\label{eq:beta_opt}
\beta_{\mathrm{opt}}  =  P(hh')^{-1} \, P(hf).
\end{equation}
Here we assume that the functions $h_1, \ldots, h_m$ are linearly independent in $L^2(P)$, so that the $m \times m$ covariance matrix $P(hh') = (P(h_jh_k))_{j,k=1}$ is invertible. 
The minimal asymptotic variance is
\begin{equation}
\label{eq:sigmaopt}
\sigma^2( f - \beta_{\mathrm{opt}}' h )
=
\sigma^2(f) - P(f h') \, P(hh')^{-1} \, P(h f).
\end{equation}

In practice, $\beta_{\mathrm{opt}}$ in \eqref{eq:beta_opt} is unknown and needs to be estimated. Any estimator $\hat{\beta}_n$ of $\beta_{\mathrm{opt}}$ produces a control variate estimator: $\hat{\mu}_n( \hat{\beta}_n ) = P_n ( f - \hat{\beta}_n' h )$. As soon as $\hat{\beta}_n \dto \beta_{\mathrm{opt}}$, then \citep[Theorem~1]{glynn+s:2002},
\begin{align}
\sqrt{n} \{ \hat{\mu}_n( \hat{\beta}_n ) - \mu \}
&\dto
\mathcal{N} \left( 0, \, \sigma^2( f - \beta_{\mathrm{opt}}' h ) \right),
\qquad n \to \infty.
\label{eq:limitlaw}
\end{align}
It is thus sufficient to estimate the vector $\beta_{\mathrm{opt}}$ consistently to obtain an integration procedure with the same asymptotic distribution as the oracle procedure $\hat{\mu}_n(\beta_{\mathrm{opt}})$.

The asymptotic variance in \eqref{eq:sigmaopt} may be estimated by the empirical variance
\begin{equation*}
	\hat{\sigma}_n^2( \hat{\beta}_n )
	= P_n [ \{ f - \hat{\beta}_n' h \}^2 ] - \{P_n[f - \hat{\beta}_n' h]\}^2 .
\end{equation*}
If $\hat{\beta}_n \dto \beta_{\mathrm{opt}}$, then, by the law of large numbers and Slutsky's lemma,
\begin{equation}
\label{eq:sigmacons}
	\hat{\sigma}_n^2( \hat{\beta}_n )
	\dto \sigma^2( f - \beta_{\mathrm{opt}}' h ),
	\qquad n \to \infty.
\end{equation}
Equations~\eqref{eq:limitlaw} and~\eqref{eq:sigmacons} justify the usual asymptotic confidence intervals for $\mu$. 

\subsection{Ordinary least squares estimator}
\label{ss:OLS}

To estimate $\beta_{\mathrm{opt}} = P(hh')^{-1} \, P(hf)$, multiple options exist \citep{glynn+s:2002}. The more common estimator is
\begin{equation*}
	\hat{\beta}_n^{\mathrm{OLS}} 
	= G_n^{-1} \bigl\{ P_n(hf) - P_n(h) \, P_n(f) \bigr\},
\end{equation*}
where $G_n = P_n(hh') - P_n(h) \, P_n(h')$ is the empirical covariance matrix of the control variates, assumed to be invertible, which is the case with large probability under the conditions in Section~\ref{sec:asym}. The resulting ordinary least squares Monte Carlo estimator is
\begin{equation*}
  \hat{\mu}_n^{\mathrm{OLS}}
= \hat{\mu}_n(\hat{\beta}_n^{\mathrm{OLS}})
=  P_n(f) - \bigl\{ P_n(fh') - P_n(f) \, P_n(h') \bigr\} \, G_n^{-1} \, P_n(h).
\end{equation*}
The OLSMC variance estimator is equal to the sample analogue of \eqref{eq:sigmaopt}:
\begin{align*}
	\hat{\sigma}_{n,\mathrm{OLS}}^2   
	&= \hat{\sigma}_n^2( \hat{\beta}_n^{\mathrm{OLS}} ) \\
	&= P_n[\{f - P_n(f) \}^2]
	-
	\{ P_n(f h') - P_n(f) \, P_n(h') \} \, G_n^{-1} \, \{ P_n(h f) - P_n(h) \, P_n(f) \}.
\end{align*}

The terminology stems from the well-known \citep{Glasserman2003} property that
\begin{equation}
\label{eq:OLS}
( \hat{\mu}_n^{\mathrm{OLS}}, \hat{\beta}_n^{\mathrm{OLS}} )
=
\argmin_{(\alpha, \beta) \in \reals \times \reals^m} \sum_{i=1}^n \{ f(X_i) - \alpha - \beta' h(X_i) \}^2.
\end{equation}
The identity \eqref{eq:OLS} is a consequence of the normal equations in the multiple linear regression model
\[
	f(X_i) = \mu + \beta_{\mathrm{opt}}' h(X_i) + \eps_i, 
	\qquad i = 1, \ldots, n,
\]
with dependent variable $f(X_i)$, explanatory variables $h_1(X_i), \ldots, h_m(X_i)$, and errors $\eps_i$. The intercept is $\mu$ whereas the vector of regression coefficients is $\beta_{\mathrm{opt}}$. The errors are $\eps_i = \eps(X_i)$ with 
$\eps = f - \mu - \beta_{\mathrm{opt}}' h \in L^2(P)$, a mean-zero function which is uncorrelated with each of the control functions, i.e., $P(\eps) = 0$ and $P(h \eps) = 0$. The variance of the errors is equal to the asymptotic variance of the OLSMC estimator: $P(\eps^2) = \sigma^2(f - \beta_{\mathrm{opt}}'h)$.

Equation~\eqref{eq:OLS} has the convenient consequence that the OLSMC estimator and the variance estimator can be computed via standard linear regression software \citep[Section~8.9]{mcbook}. Also, it implies that the OLSMC integration rule integrates the constant function and the $m$ control functions exactly.

\subsection{Orthogonal projections}

Geometric considerations lead to another, insightful representation of the OLSMC estimator, revealing properties relevant for asymptotic theory. Let $H^{(n)}$ be the $n \times m$ matrix
\begin{equation}
\label{eq:hn}
H^{(n)} = 
\begin{pmatrix}
h_1(X_1) & \ldots & h_m(X_1) \\
\vdots & & \vdots \\
h_1(X_n) & \ldots & h_m(X_n)
\end{pmatrix}.
\end{equation}
Let $\Pi_{n,m}$ be the $n \times n$ projection matrix on the column space of the matrix $H^{(n)}$ in \eqref{eq:hn}. If the $m$ columns of $H^{(n)}$ are linearly independent, then 
\begin{equation}
\label{eq:Pinm}
\Pi_{n,m} 
= H^{(n)} \{ (H^{(n)})' H^{(n)} \}^{-1} (H^{(n)})'
= n^{-1} H^{(n)} \, P_n(hh')^{-1} \, (H^{(n)})',
\end{equation}
the so-called hat matrix in a multiple linear regression model without intercept on the $m$ variables $(h_j(X_i))_{i=1}^n$, $j = 1, \ldots, m$. Even if the $m$ columns of $H^{(n)}$ are not linearly independent, the projection matrix $\Pi_{n,m}$ is well-defined, for instance, by using Moore--Penrose inverses.

Write the OLSMC estimator in \eqref{eq:OLS} in terms of two nested minimization problems:
\[
\hat{\mu}_n^{\mathrm{OLS}}
=
\argmin_{\alpha \in \reals} 
\left[ 
\min_{\beta \in \reals^m} 
\sum_{i=1}^n \{ f(X_i) - \alpha - \beta' h(X_i) \}^2
\right].
\]
Given $\alpha \in \reals$, the minimum over $\beta \in \reals^m$ is well-defined and is attained as soon as $\beta$ satisfies $ H^{(n)} \beta = \Pi_{n,m} (f^{(n)} - \alpha 1_n)$, where $f^{(n)} = (f(X_1), \ldots, f(X_n))'$ and where $1_n$ is an $n \times 1$ vector with all elements equal to $1$. We find that
\begin{equation}
\label{eq:mu_OLS:nu}
\hat{\mu}_n^{\mathrm{OLS}}
=
\argmin_{\alpha \in \reals} 
| (I_n - \Pi_{n,m})(f^{(n)} - \alpha 1_n) |^2,
\end{equation}
where $| v | = (v' v)^{1/2}$ is the Euclidean norm of a vector $v$ and $I_n$ is the $n \times n$ identity matrix. It follows that $\alpha (I_n - \Pi_{n,m}) 1_n$ is equal to the orthogonal projection of $(I_n - \Pi_{n,m})f^{(n)}$ on the line passing through the origin and $(I_n - \Pi_{n,m}) 1_n$. A necessary and sufficient condition for the uniqueness of $\alpha \in \reals$ is that $(I_n - \Pi_{n,m}) 1_n$ is not equal to the zero vector, that is, $1_n$ is \emph{not} an element of the column space of $H^{(n)}$. Suppose this condition holds. 
Then
$1_n'(I_n - \Pi_{n,m})1_n = | (I_n - \Pi_{n,m}) 1_n |^2 > 0$ and
\begin{equation}
\label{eq:mu_OLS:Pi}
\hat{\mu}_n^{\mathrm{OLS}}
=
\frac{(f^{(n)})' (I_n - \Pi_{n,m}) 1_n}{1_n' (I_n - \Pi_{n,m}) 1_n}.
\end{equation}

If, in addition, the columns of $H^{(n)}$ are linearly independent, then, by \eqref{eq:Pinm},
\begin{equation}
\label{eq:mu_OLS:2}
\hat{\mu}_n^{\mathrm{OLS}}
=
\frac%
{P_n(f) - P_n(fh') \, P_n(hh')^{-1} \, P_n(h)}%
{1 - P_n(h') \, P_n(hh')^{-1} \, P_n(h)}.
\end{equation}
Indeed, we have $(f^{(n)})' 1_n = n \, P_n(f)$, $(f^{(n)})' H^{(n)} = n \, P_n(f h')$, and $1_n' H^{(n)} = n \, P_n(h')$. 

We have supposed that the $n \times 1$ vector $1_n$ is \emph{not} an element of the column space of $H^{(n)}$. If it is, then there obviously cannot exist a weight vector such that the corresponding linear integration rule integrates both the constant functions and the control functions exactly. Also, the minimizer $\alpha$ in \eqref{eq:OLS} is then no longer identifiable. In that case, we recommend to reduce the number of control functions. Actually, when $m$ is not too large with respect to $n$ (Section~\ref{sec:asym}), the denominator in \eqref{eq:mu_OLS:2} tends to $1$ in probability, implying that, with probability tending to one, $1_n$ is \emph{not} an element of the column space of $H^{(n)}$.

The representation \eqref{eq:mu_OLS:Pi} also implies that the OLSMC estimator does not change if we replace the vector $h$ of control functions by the vector $A h$, where $A$ is an arbitrary invertible $m \times m$ matrix. Indeed, such a transformation results in changing the matrix $H^{(n)}$ in \eqref{eq:hn} into $H^{(n)} A'$, but both $n \times m$ matrices share the same column space.

The OLSMC variance estimator  $\hat{\sigma}_{n,\mathrm{OLS}}^2$ coincides with $n^{-1}$ times the minimal sum of squares in \eqref{eq:OLS} and \eqref{eq:mu_OLS:nu}:
\begin{equation}
\label{eq:sigma:proj}
\hat{\sigma}_{n,\mathrm{OLS}}^2
=
\frac{1}{n} 
(f^{(n)} - \hat{\mu}_n^{\mathrm{OLS}} 1_n)'
(I_n - \Pi_{n,m}) 
(f^{(n)} - \hat{\mu}_n^{\mathrm{OLS}} 1_n).
\end{equation}
Recall $f = \mu + \beta_{\mathrm{opt}}' h + \eps$, where $\eps \in L^2(P)$ is centered and uncorrelated with all control functions $h_j$. If $P_n(hh')$ is invertible and $P_n(h') P_n(hh')^{-1} P_n(h) < 1$, we can use \eqref{eq:Pinm} for $\Pi_{n,m}$ and \eqref{eq:mu_OLS:2} for $\hat{\mu}_{n}^{\mathrm{OLS}}$ to work out \eqref{eq:sigma:proj} and find (proof in Section~\ref{sec:proofs})
\begin{multline}
\label{eq:sigma:handy}
\hat{\sigma}_{n,\mathrm{OLS}}^2
=
P_n(\eps^2) - P_n( \eps h' ) \, P_n(h h')^{-1} \, P_n(h \eps) \\
- 
(\hat{\mu}_n^{\mathrm{OLS}} - \mu)^2 \{1 - P_n(h') P_n(hh')^{-1} P_n(h)\}.
\end{multline}
Since $\hat{\sigma}_{n,\mathrm{OLS}}^2 \le P_n(\eps^2)$ and $\expec\{P_n(\eps^2)\} = \sigma^2$, it follows that $\hat{\sigma}_{n,\mathrm{OLS}}^2$ has a negative bias. In view of the multiple linear regression perspective in Section~\ref{ss:OLS} and to possibly reduce this bias, one may prefer to multiply the variance estimator by $n / (n-m-1)$, although this particular correction is justified only in case of a linear model with fixed design and centered, uncorrelated, and homoskedastic Gaussian errors.

\section{Central limit theorem for a growing number of control variates}
\label{sec:asym}

By~\eqref{eq:limitlaw}, the asymptotic variance of the OLSMC estimator $\hat{\mu}_n^{\mathrm{OLS}}$ of $\mu = P(f)$ with a fixed number of control variates is equal to the variance of the error variable 
\begin{equation}
\label{eq:error_variable}
	\eps = f - \mu - \beta_{\mathrm{opt}}' h, 
\end{equation}
where $\mu + \beta_{\mathrm{opt}}' h$ is the orthogonal projection in $L^2(P)$ of $f$ on the linear space $\mathcal{F}_m$ spanned by $\{1, h_1, \ldots, h_m\}$. Suppose that the number, $m = m_n$, of control functions varies with $n$ and tends to infinity and that $f$ can be written as an $L^2(P)$ limit of a sequence of approximating functions in $\mathcal{F}_{m_n}$. Then $\sigma_n^2 = P(\eps_n^2) \to 0$ as $n \to \infty$, where $\eps_n$ is the error variable $\eps$ in \eqref{eq:error_variable} when there are $m_n$ control variates in use. Then we may hope that the asymptotic variance of the OLSMC estimator becomes zero too, so that its convergence rate is $\oh_{\pr}(1/\sqrt{n})$, faster than the one of the Monte Carlo estimator. More precisely, we may hope to pin the convergence rate down to $\Oh_{\pr}( \sigma_n / \sqrt{n} )$.

\subsection{Set-up}
\label{subsec:set-up}

Our set-up is a triangular array of control functions. Let $h_n = (h_{n,1}, \ldots, h_{n,m_n})'$ for some positive integer sequence $m_n \to \infty$, where $h_{n,j} \in L^2(P)$ and $P(h_{n,j}) = 0$ for all $n$ and $j$. Assume that $h_{n,1}, \ldots, h_{n,m_n}$ are linearly independent in $L^2(P)$, so that the $m_n \times m_n$ Gram matrix $P(h_n h_n') = \left( P(h_{n,j} h_{n,k}) \right)_{j,k}$ is invertible. Examples of control functions we have in mind are polynomials or trigonometric functions, in which case a single sequence $h_1, h_2, \ldots$ would suffice, or spline functions on an interval with the knots forming a grid depending on $m_n$, an example which requires a triangular array of control functions.

There is no additional mathematical cost to let the integrands depend on $n$ as well: we want to calculate the integral $\mu_n = P(f_n)$ of $f_n \in L^2(P)$. In doing so, we obtain results that are locally uniform in the integrand. We have $f_n = \mu_n + \beta_n' h_n + \eps_n$ for some vector $\beta_n \in \reals^{m_n}$ determined by the orthogonality equations $P(\eps_n h_{n,j}) = 0$ for all $j = 1, \ldots, m_n$. We have $P(\eps_n) = 0$, while the error variance is $\sigma_n^2 = P(\eps_n^2)$. To avoid trivialities, we assume that $\sigma_n^2 > 0$, that is, $f_n$ is not equal to a constant plus a linear combination of the control functions, in which case its integral would be known. Of particular interest is the case where $\sigma_n^2 \to 0$ as $n \to \infty$, although we do not impose this.


\subsection{Leverage condition}

Consider the linear regression model without intercept term for the centered integrand on the control variates:
\[
	f_n(X_i) - \mu_n = h_n'(X_i) \beta_n + \eps_n(X_i),
	\qquad i = 1, \ldots, n.
\]
The $n \times m_n$ design matrix is $H^{(n)}$ in \eqref{eq:hn}, whereas the $n \times n$ projection matrix onto the column space of $H^{(n)}$ is $\Pi_n \equiv \Pi_{n,m_n}$ in \eqref{eq:Pinm}, assuming that the $m_n$ columns of $H^{(n)}$ are linearly independent. In multiple linear regression theory, this projection matrix is called the \emph{hat matrix}, and its $i$th diagonal element is called the \emph{leverage} of the $i$th sample point:
\[
	\Pi_{n,ii} = n^{-1} h_n(X_i)' \, P_n(h_n h_n')^{-1} \, h_n(X_i), 
	\qquad i = 1, \ldots, n.
\]
The average leverage is equal to $n^{-1} \trace(\Pi_n) = m_n/n$. Points for which $\Pi_{n,ii} > c m_n/n$ for some pre-determined constant $c > 1$, often $c = 2$ or $c = 3$, are commonly flagged as high-leverage points; see \cite{velleman+w:1981} and the references therein.

We have $\Pi_{n,ii} = n^{-1} \hat{q}_n(X_i)$ for $i = 1, \ldots, n$, where $\hat{q}_n(x) = h_n(x)' \, P_n(h_n h_n')^{-1} \, h_n(x)$ is the sample version of what could be called the \emph{leverage function}
\begin{equation}
\label{eq:q}
	q_n(x) = h_n(x)' \, P(h_n h_n')^{-1} \, h_n(x), 
	\qquad x \in S.
\end{equation}
Note that $q_n(x)$ is the squared Mahalanobis distance of $h_n(x)$ to the center $P(h_n) = 0$ of the distribution of the $m_n$-dimensional random vector $h_n$ under $P$. The expectation of the leverage function is equal to the dimension of the control space,
\begin{equation}
\label{eq:Pq}
P(q_n) = m_n.
\end{equation}
Recall that the OLSMC estimator does not change if we replace the vector $h_n$ by the vector $A h_n$, where $A$ is any invertible $m_n \times m_n$ matrix. The function $q_n$ is invariant under such transformations of the control functions, as can be easily checked. It follows that $q_n$ is linked to the linear space spanned by the control functions $h_{n,1}, \ldots, h_{n,m_n}$ rather than to the functions themselves.

To establish the rate of convergence of the OLSMC estimator, we need to prohibit the occurrence of points of which the leverage is too high. The criterion commonly used in regression diagnostics to flag high-leverage points would suggest that we impose that $\sup_{x \in S} q_n(x) = \Oh(m_n)$ as $n \to \infty$. [By \eqref{eq:Pq}, a smaller bound can never be satisfied.] Instead, we impose a weaker condition, which is reminiscent of Assumption~2(ii) in \cite{newey:1997}.

\begin{condition}[Leverage]
	\label{cond:q}
	We have
	\begin{equation}
	\label{eq:q:suff}
	\sup_{x \in S} q_n(x) = \oh(n/m_n), \qquad n \to \infty.
	\end{equation}
\end{condition}
Equations~\eqref{eq:Pq} and~\eqref{eq:q:suff} imply
\begin{equation}
\label{eq:on}
	P(q_n^2) = \oh(n), \qquad n \to \infty.
\end{equation}
Since $m_n^2 = P(q_n)^2 \le P(q_n^2)$, Equation~\eqref{eq:on} implies that $m_n = \oh(n^{1/2})$, restricting the dimension of the control space. As a consequence, also $m_n = \oh(n/m_n)$, meaning that Equation~\eqref{eq:q:suff} is indeed weaker than $\sup_{x \in S} q_n(x) = \Oh(m_n)$ as $n \to \infty$.

According to \cite{huber:1981}, the reciprocal of the leverage can be seen as the equivalent number of observations entering into the determination of the predicted response for the $i$th point. Since our condition implies that $\sup_{x \in S} n^{-1} q_n(x) = \oh(1/m_n)$ as $n \to \infty$, a possible interpretation of Condition~\ref{cond:q} is that the equivalent number of observations used to predict each response is of larger order than the number of control variates, $m_n$.




\subsection{Main results}

Assume the set-up of Section~\ref{subsec:set-up}.

\begin{theorem}[Rate]
	\label{thm:rate}
	If Condition~\ref{cond:q} holds, then, as $n \to \infty$, the OLSMC estimator is well-defined with probability tending to one and
	\begin{equation}
	\label{eq:AN:prob}
	\frac{\sqrt{n}}{\sigma_{n}} \left( \hat{\mu}_n^{\mathrm{OLS}} - \mu_n \right)
	=
	\frac{\sqrt{n}}{\sigma_{n}} P_n( \eps_n ) + \oh_{\pr}(1)
	=
	\Oh_{\pr}(1).
	\end{equation}
	In particular, $\hat{\mu}_n^{\mathrm{OLS}} - \mu_n = \Oh_{\pr}( \sigma_{n} / \sqrt{n} )$ as $n \to \infty$.
\end{theorem}

To prove asymptotic normality of the estimation error, we apply the Lindeberg--Feller central limit theorem. The Lindeberg condition, which is both necessary and sufficient \cite[Theorem~5.12]{kallenberg:2002}, also guarantees consistency of the OLS variance estimator. A sufficient but not necessary condition as well as some intuition are provided in Remark~\ref{rem:Lindeberg} below. Recall that the arrow $\dto$ denotes weak convergence.

\begin{condition}[Lindeberg]
	\label{cond:Lindeberg}
	For every $\delta > 0$, we have, as $n \to \infty$,
	\[
	P[ (\eps_n/\sigma_n)^2 \, \1 \{ \lvert\eps_n/\sigma_n\rvert > \delta \sqrt{n} \} ] = \oh(1).
	\]
\end{condition}

\begin{theorem}[Asymptotic normality]
	\label{thm:AN}
	Suppose Condition~\ref{cond:q} holds. Then Condition~\ref{cond:Lindeberg} holds if and only if
	\begin{equation}
	\label{eq:AN}
	\frac{\sqrt{n}}{\sigma_{n}} \left( \hat{\mu}_n^{\mathrm{OLS}} - \mu_n \right) \dto \mathcal{N}(0, 1),
	\qquad n \to \infty.
	\end{equation}
	Moreover, under Conditions~\ref{cond:q} and~\ref{cond:Lindeberg}, the variance estimator is consistent in the sense that
	\begin{equation} 
	\label{eq:sigma:consistent}
	\hat{\sigma}_{n,\mathrm{OLS}}^2 / \sigma_n^2 \dto 1, \qquad n \to \infty.
	\end{equation}
	Equation~\eqref{eq:AN} thus remains true if $\sigma_n$ is replaced by $\hat{\sigma}_{n,\mathrm{OLS}}$.
\end{theorem}

Theorem~\ref{thm:AN} justifies the use of the usual asymptotic confidence intervals of nominal coverage $1-\alpha$ of the form $\hat{\mu}_{n,\mathrm{OLS}} \pm z_{1-\alpha/2} \, \hat{\sigma}_{n,\mathrm{OLS}} / \sqrt{n}$, where $z_p$ is the $p$th quantile of the standard normal distribution. As in multiple linear regression, quantiles of the Student $t$ distribution with $n-m_n-1$ degrees of freedom may be used instead, making the intervals a bit wider, although there is no guarantee that this will bring the real coverage closer to the nominal one when the errors are not normally distributed.

\subsection{Discussion}

\begin{remark}[Weakening the leverage condition]
	Equation~\eqref{eq:q:suff} implies
	\begin{equation}
	\label{eq:onm}
	P(q_n \eps_n^2) = \oh\{ (n/m_n) \sigma_n^2 \},
	\qquad n \to \infty.
	\end{equation}
	In fact, Theorems~\ref{thm:rate} and~\ref{thm:AN} would remain true if Condition~\ref{cond:q} would be replaced by the weaker pair of Equations~\eqref{eq:on} and~\eqref{eq:onm}. In addition, by the Cauchy--Schwarz inequality, $P(q_n \eps_n^2) \le P(q_n^2)^{1/2} P(\eps_n^4)^{1/2}$, so that Equation~\eqref{eq:on} together with
	\begin{equation} 
	\label{eq:eps4}
		P(\eps_n^{4}) = \Oh\{ (n/m_n^2) \sigma_n^4 \},
		\qquad n \to \infty,
	\end{equation}
	would be sufficient. However, both Equations~\eqref{eq:onm} and~\eqref{eq:eps4} depend on the integrand through the error function $\eps_n$ and may be difficult to check. The advantage of Equation~\eqref{eq:q:suff} is that it only depends on the control variates and not on the integrand.
\end{remark}

\begin{remark}[Checking the leverage condition]
	\label{rem:leverage}
	When calculating $q_n$ is complicated, the following bound may be helpful in establishing \eqref{eq:q:suff}: we have 
	$$
		q_n 
		\le \lambda_{n,1}^{-1} h_n' h_n 
		= \lambda_{n,1}^{-1} \sum_{j=1}^{m_n} h_{n,j}^2,
	$$
	where $\lambda_{n,1} > 0$ is the smallest eigenvalue of $P(h_n h_n')$. See also Section~\ref{sec:examples} for a number of examples in which we check the leverage condition.
\end{remark}

\begin{remark}[On the Lindeberg condition]
	\label{rem:Lindeberg}
	As already mentioned, Condition~\ref{cond:Lindeberg} is both necessary and sufficient for \eqref{eq:AN} to hold. In view of H\"{o}lder's inequality, the condition is implied by the Lyapunov condition that there exists $\eta > 0$ such that 
	\[ 
		\sup_{n \ge 1} P[ \lvert \eps_n/\sigma_n \rvert^{2+\eta}] < \infty.
	\]
	The latter condition is equivalent to $\lVert \eps_n \rVert_{2+\eta} = \Oh( \lVert \eps_n \rVert_{2} )$ as $n \to \infty$, where $\lVert \, \cdot \, \rVert_{p}$ denotes the $L_p(P)$ (semi-)norm. 
	
	Intuitively, the Lindeberg condition requires that the error sequence $\eps_n$ behaves regularly in some sense. It fails for instance if, along a subsequence, the centered integrand $f_n - P(f_n)$ is a linear combination of the control functions $h_{n,1}, \ldots, h_{n,m_n}$: if the fit is perfect ($\sigma_n = 0$), the integration error is zero and cannot be normalized to be asymptotically standard Gaussian. See Example~\ref{ex:Liouville} in Section~\ref{sec:examples} below for an illustration on checking the Lindeberg condition.
\end{remark}

\section{Computational cost}
\label{sec:cost}

For pure Monte Carlo integration, the main computational cost stems from the $n$ evaluations of the integrand $f$. The computation time is therefore of the order $\Oh(n)$. Decreasing the integration error then simply amounts to increase the number, $n$, of random evaluation points $X_i$.

Another way to improve the integration accuracy is by increasing the number of control variates. For fixed sample size $n$, this will decrease the standard deviation $\sigma_n = \{P(\eps_n^2)\}^{1/2}$ of the error term $\eps_n \in L^2(P)$ in the representation
\[
	f = \mu + \beta_{n,1} h_{n,1} + \cdots + \beta_{n,m_n} h_{n,m_n} + \eps_n,
\]
with $\beta_n \in \reals^{m_n}$ determined by $P(\eps_n) = 0$ and $P(\eps_n h_{n,j}) = 0$ for all $j = 1, \ldots, m_n$. 

However, the use of $m_n$ control variates makes the number of operations go up to $\Oh(nm_n^2)$. The bottleneck comes from the $m_n \times m_n$ empirical Gram matrix $P_n(h_nh_n')$, each element of which requires calculating an arithmetic mean over the $n$ sample points. The other terms in \eqref{eq:mu_OLS:2} require fewer operations. Indeed, evaluating the $m_n$ control variates $h_{n,j}$ in the $n$ sample points $X_i$ amounts to $\Oh(nm_n)$ operations. The vectors $P_n(h_n)$ and $P_n(fh_n')$ contain $m_n$ elements, each of which is an arithmetic mean over the Monte Carlo sample, requiring $\Oh(nm_n)$ operations too. The matrix inversion and matrix multiplication in \eqref{eq:mu_OLS:2} represent $\Oh(m_n^3)$ operations. Since necessarily $m_n^2 = \oh(n)$ by~\eqref{eq:Pq} and~\eqref{eq:on}, the latter represents an additional cost of only $\oh(nm_n)$ operations.

The method of control variates thus invests $\Oh(nm_n^2)$ operations to achieve an asymptotic standard deviation of $\sigma_n n^{-1/2}$. Alternatively, one could allocate all computation resources to augmenting the Monte Carlo sample size from $n$ to $nm_n^2$, yielding a standard deviation of the order $\Oh(n^{-1/2} m_n^{-1})$. At equal computational budget, the method of control variates with the number of control variates tending to infinity will thus converge at a faster rate than naive Monte Carlo integration as soon as
\begin{equation}
\label{eq:whenfaster}
	\sigma_n = \oh(m_n^{-1}), \qquad n \to \infty. 
\end{equation}
Whether or not this is the case depends on the control variates and the integrand; see the examples in Section~\ref{sec:examples}.

For certain families of control variates, the computational cost can be brought down from $\Oh(nm_n^2)$ to $\Oh(nm_n)$. This is the case for instance for the normalized indicator functions in Example~\ref{ex:indicator} below and more generally for control variates that arise from functions that, prior to centering, have localized supports, such as splines or wavelets. In such cases, only $\Oh(m_n)$ elements of the $m_n \times m_n$ matrix $P_n(h_n h_n')$ vary with the sample, while the other elements are known in advance and thus non-random. Comparing the asymptotic standard deviation $\sigma_n n^{-1/2}$ of the control variate error with the one of the naive Monte Carlo method at sample size $n m_n$, which is $\Oh(n^{-1/2} m_n^{-1/2})$, we find that, at equal computational budget, the OLSMC estimator already converges at a faster rate than the Monte Carlo estimator as soon as
\begin{equation}
\label{eq:whenfaster2}
	\sigma_n = \oh(m_n^{-1/2}), \qquad n \to \infty.
\end{equation}

If evaluating the integrand $f$ is expensive while evaluating the control functions $h_j$ is cheap, then, in practice, it may still be computationally beneficial to increase the number of control variates rather than the Monte Carlo sample size, even though this is not backed up by the asymptotic considerations so far.

Computational benefits can also occur when there are multiple integrands. Indeed, it is well known that the method of control variates can be seen as a form of weighted Monte Carlo, i.e.,
\[
	\hat{\mu}_n^{\mathrm{OLS}}
	=
	\textstyle{\sum_{i=1}^n} w_{n,i} f(X_i)
\]
where the expression of the weight vector $w_n \in \reals^n$ can for instance be deduced from \eqref{eq:mu_OLS:Pi}; see also \citep[eq.~(4.20)]{Glasserman2003}. The control variates only enter the formula through these weights, which, even in case of multiple integrands, thus need to be computed only once. This feature can for instance be put to work to efficiently estimate quantiles \citep{HesterbergNelson1998}, price financial options \citep{glasserman+y:2005}, and compute likelihoods arising in statistical models with latent variables \citep{portier+s:2018v3}. 

\section{Examples}
\label{sec:examples}

\begin{example}[Post-stratification]
	\label{ex:indicator}
	On $S = [0, 1]$ equipped with the Lebesgue measure $P$, let $h_{n,j}(x) = (m_n+1) \1\{ x \in \mathcal{I}_{m_n,j} \} - 1$ for $j = 1, \ldots, m_n$, where $\mathcal{I}_{m_n,j} = [(j-1)/(m_n+1), j/(m_n+1))$. The control variates are normalized indicator functions induced by a partition of $[0, 1]$ into $m_n + 1$ intervals of equal length. Note that the last cell $\mathcal{I}_{m_n,m_n+1} = [m_n / (m_n+1), 1]$ is omitted, since its normalized indicator $h_{n,m_n+1}$ is a linear combination of $h_{n,1}, \ldots, h_{n,m_n}$.
	
	Unless one or more cells contain no sample points $X_i$, the constant vector $1_n$ is not an element of the column space of the design matrix in \eqref{eq:hn} and the OLSMC estimator is well-defined. A particular cell being empty with probability $\{1 - (m_n+1)^{-1}\}^n$, the probability that at least one cell is empty is bounded by $(m_n + 1) (1 - (m_n+1)^{-1})^n$, which converges to zero as soon as $m_n \ln(m_n) = \oh(n)$.
	
	The Gram matrix $P(h_n h_n') = (m_n+1) I_{m_n} - 1_{m_n} 1_{m_n}'$ has inverse $P(h_n h_n')^{-1} = (m_n+1)^{-1} (I_{m_n} + 1_{m_n} 1_{m_n}')$. The function $q_n = h_n' P(h_n h_n')^{-1} h_n = \{ h_n' h_n + (h_n' 1_{m_n})^2 \} / (m_n+1) = m_n$ is constant. Condition~\ref{cond:q} is satisfied as soon as $m_n = \oh(n^{-1/2})$. 
	
	Let $f_{m_n,j} = (m_n+1)^{-1} P( f \, \1\{ \, \cdot \in \mathcal{I}_{m_n,j} \} )$ be the average of the integrand $f$ on the cell $\mathcal{I}_{m_n,j}$, for $j = 1, \ldots, m_n+1$. The OLSMC estimator is equal to the arithmetic mean of the Monte Carlo estimates of these $m_n + 1$ local averages $f_{m_n,j}$. This is also the value obtained by post-stratification \citep[Example~8.4]{mcbook}. The number of operations required to calculate the OLSMC estimator is thus $\Oh(nm_n)$ only.
	
	The projection of $f$ on the space spanned by $\{1, h_{n,1}, \ldots, h_{n,m_n}\}$ is equal to the piecewise constant function $f^{(n)}$ with value $f_{m_n,j}$ on $\mathcal{I}_{m_n,j}$ for $j = 1, \ldots, m_n+1$. If $f$ is Lipschitz, then the error term $\eps_n = f - f^{(n)}$ will satisfy $\sup_{x \in S} \lvert \eps_n(x) \rvert = \Oh(m_n^{-1})$. In particular, $\sigma_n = \Oh(m_n^{-1})$. If, in addition, $\liminf_{n \to \infty} \sigma_n m_n > 0$, then $\eps_n / \sigma_n$ remains bounded uniformly, and the Lindeberg condition (Condition~\ref{cond:Lindeberg}) is satisfied too.
	
	The standard deviation of the OLSMC error at sample size $n$ is $\sigma_n n^{-1/2}$, achieved at $\Oh(nm_n)$ operations, while the one of the Monte Carlo integration error at sample size $nm_n$ is $n^{-1/2}m_n^{-1/2}$. For Lipschitz functions, we have $\sigma_n = \Oh(m_n^{-1}) = \oh(m_n^{-1/2})$, as in \eqref{eq:whenfaster2}. At comparable computational budgets, the OLSMC estimator thus achieves a faster rate of convergence than the Monte Carlo estimator.
	
	On the $d$-dimensional cube $S = [0, 1]^d$, we can employ a similar construction, starting from a partition of $S$ into $\Oh(m_n)$ cubes with side length $\Oh(m_n^{1/d})$. For Lipschitz functions, the error term $\eps_n$ will then have a standard deviation $\sigma_n$ of the order $\Oh(m_n^{-1/d})$. As soon as $d \ge 2$, Equation~\eqref{eq:whenfaster2} is no longer fulfilled. Given a comparable number of operations, the OLSMC estimator cannot achieve a convergence rate acceleration in comparison to ordinary Monte Carlo integration. \hfill $\bigtriangleup$
\end{example}

\begin{example}[Lindeberg condition]
	\label{ex:Liouville}
	We elaborate on Example~\ref{ex:indicator} to illustrate the Lindeberg condition. For ease of notation, put $k_n = m_n + 1$ and consider the integrand $f(x) = \1_{[u,1]}(x)$ for $x \in [0, 1]$, for some fixed $u \in [0, 1]$. 
	
	If $u$ is rational, then for infinitely many integer $n$ we can write $u = \ell_n / k_n$ for some $\ell_n \in \{0, \ldots, k_n\}$, and it follows that $f$ is a member of the linear span $\mathcal{F}_n$ of $\{1, h_{n,1}, \ldots, h_{n,m_n}\}$. In that case, $\sigma_n = 0$ for such $n$, and the normalized integration can obviously not converge to the standard normal distribution.
	
	Suppose that $u \in (0, 1)$ is irrational, and for every $n$, let $\ell_n \in \{0, \ldots, k_n-1\}$ be such that $a_n = \ell_n / k_n \le u < (\ell_n+1)/k_n = b_n$. The $L_2$-orthogonal projection of $f$ on $\mathcal{F}_n$ is given by the piecewise constant function
	\[
		f^{(n)}(x) =
		\begin{cases}
			0 & \text{if $0 \le x < a_n$,} \\
			v_n = k_n (b_n - u) & \text{if $x \in [a_n, b_n)$,} \\
			1 & \text{if $b_n \le x \le 1$.}
		\end{cases}
	\]
	The approximation error $\eps_n = f - f^{(n)}$ is
	\[
		\eps_n(x) =
		\begin{cases}
			0 & \text{if $x \in [0, 1] \setminus [a_n, b_n)$,} \\
			-v_n & \text{if $a_n \le x < u$,} \\
			1-v_n & \text{if $u \le x < b_n$,}
		\end{cases}
	\]
	with error variance $\sigma_n^2 = P(\eps_n^2) = k_n (b_n - u) (u - a_n)$. The squared, standardized approximation error is thus
	\[
		\eps_n^2(x) / \sigma_n^2
		=
		\begin{cases}
			0 & \text{if $x \in [0, 1] \setminus [a_n, b_n)$,} \\
			k_n \frac{b_n - u}{u - a_n} & \text{if $a_n \le x < u$,} \\
			k_n \frac{u - a_n}{b_n - u} & \text{if $u \le x < b_n$.}
		\end{cases}
	\]
	
	Now assume that there exists $c > 0$ such that for all pairs of integers $(p, q)$ with $q \ge 1$, we have
	\[
		\left\lvert u - \frac{p}{q} \right\rvert > \frac{c}{q^2}.
	\]
	Such a number $u$ is called a \emph{badly approximable number} \citep[p.~245]{bugeaud:2012}. It then follows that
	\[
		\sup_{x \in [0, 1]} \eps_n^2(x) / \sigma_n^2
		\le k_n \frac{k_n^{-1}}{c k_n^{-2}} = \frac{1}{c} k_n^2.
	\] 
	Since necessarily $k_n^2 = \oh(n)$ by the leverage condition, it follows that the indicator in the Lindeberg condition is zero for all sufficiently large $n$, and thus that the Lindeberg condition is fulfilled.
\end{example}


\begin{example}[Univariate polynomials]
	\label{ex:Legendre}
	Suppose that $h_{n,j} = h_j$ is equal to the Legendre polynomial $L_j$ of degree $j = 1, \ldots, m_n$. The Legendre polynomials are orthogonal on $S = [-1, 1]$ with respect to the uniform distribution $P$. The Gram matrix $P(h_n h_n')$ is diagonal with entries $1 / (2j+1)$ on the diagonal. Furthermore, the Legendre polynomials satisfy $|L_j(x)| \le 1$ for $x \in [-1, 1]$ while $L_j(1) = 1$. Hence $q_n(x) = \sum_{j=1}^{m_n} (2j+1) L_j(x)^2$, with supremum $q_n(1) = \sum_{j=1}^{m_n} (2j+1) = m_n(m_n+2)$. Equation~\eqref{eq:q:suff} is satisfied when $m_n = \oh(n^{1/3})$.
	
	If $f$ is $k+1$ times continuously differentiable for some integer $k \ge 1$, then the bounds on the Legendre coefficients in Theorem~2.1 in \cite{wang+x:2012} imply that $\sigma_n^2 = \Oh(m_n^{-2k-1})$. The convergence rate of the OLSMC estimator is thus $\Oh(m_n^{-k-1/2} n^{-1/2})$. The smoother $f$, the faster the rate. Condition~\eqref{eq:whenfaster} is fulfilled as soon as $f$ is twice continuously differentiable ($k \ge 1$). For such functions $f$, increasing the number of polynomial control variates reduces the integration error at a faster rate than increasing the number of Monte Carlo points can achieve. \hfill $\bigtriangleup$
\end{example}

For the Fourier basis on $S = [0, 1]$, it is shown in \citep{portier+s:2018v3} that essentially the same conclusions hold as for the polynomial basis in Example~\ref{ex:Legendre}.

\begin{example}[Multivariate polynomials]
	\label{ex:Legendre:multi}
	As in \cite{andrews:1991} and \cite{newey:1997}, suppose that $S = [-1, 1]^d$ (or more generally a Cartesian product of compact intervals) and that $P$ is the uniform distribution on $S$. As control variates $h_{n,j} = h_{j} : S \to \reals$, consider tensor products $h_j(x) = \prod_{\ell = 1}^d \bar{L}_{a_j(\ell)}(x_\ell)$ for $x = (x_1, \ldots, x_d) \in S$, where $\bar{L}_{a}$ is the normalised Legendre polynomial of degree $a \in \mathbb{N} = \{0, 1, 2, \ldots\}$. The sequence of degree vectors $a_j = (a_j(1), \ldots, a_j(d)) \in \mathbb{N}^d \setminus \{(0, \ldots, 0)\}$ is such that no polynomial of degree $a+1$ appears in one of the coordinates as long as not all polynomials of degree up to $a$ have appeared in all other coordinates.
	
	As shown in \cite[Example~II]{andrews:1991} and the proof of Theorem~4 in \cite{newey:1997}, we then have $\sup_{x \in S} \lvert h_j(x) \rvert = \Oh(j^{1/2})$ as $j \to \infty$ and the smallest eigenvalue of the $m \times m$ Gram matrix of $(h_1, \ldots, h_{m})$ is bounded away from zero, uniformly for all $m$. By Remark~\ref{rem:leverage}, it then follows that $\sup_{x \in S} q_n(x) = \Oh(\sum_{j=1}^{m_n} j) = \Oh(m_n^2)$. As a consequence, the leverage condition is satisfied as soon as $m_n^2 = \oh(n/m_n)$, i.e., $m_n^3 = \oh(n)$ as $n \to \infty$.
	
	Further, assume that the integrand $f = f_n$ is $k$ times continuously differentiable on $S$, for some integer $k \ge 1$. In the proof of Theorem~4 in \cite{newey:1997}, Theorem~8 in \cite{lorentz:1986} is cited according to which we have $\sup_{x \in S} \lvert \eps_n(x) \rvert = \Oh(m_n^{-k/d})$. But then also $\sigma_n = \Oh(m_n^{-k/d})$. The convergence rate of the OLSMC estimator is then $\Oh(m_n^{-k/d} n^{-1/2})$. In view of Equation~\eqref{eq:whenfaster}, it is more efficient to increase the number of control variates than the Monte Carlo sample size as soon as $k > d$, i.e., the integrand $f$ is sufficiently smooth.
\end{example}

\section{Concluding remarks}
\label{sec:discussion}

The paper provides a new asymptotic theory for Monte Carlo integration with control variates. Our main result is that the $n^{-1/2}$ convergence rate of the basic Monte Carlo method can be improved when using a growing number, $m$, of control variates. The obtained convergence rate, $n^{-1/2}\sigma_m$, is then impacted by the value of $\sigma_m$, which reflects the approximation quality of the integrand in the space of control variates. The considered examples have shown that the practical benefits might be important depending, obviously, on $\sigma_m$ and also on the computation time needed to invert the Gram matrix of the control variates. Attractive avenues for further research are now discussed. 

\paragraph{Combination with other integration methods.}
Theorem~\ref{thm:rate} echoes other studies (based on different techniques than control variates) that establish acceleration of the standard Monte Carlo rate $n^{-1/2}$. This includes Quasi-Monte Carlo integration \citep{dick+p:2010}, Gaussian quadrature \citep{brass+p:2011}, which has been studied recently in a (repulsive) Monte Carlo sampling context \citep{bardenet+h:2016}, parametric \citep{delyon+:p:2018} and nonparametric \citep{zhang:1996} adaptive importance sampling, and kernel smoothing methods \citep{delyon+p:2016}. Combining control variates with some of the previous methods, as has been done with Quasi-Monte Carlo in \citep{oates+g:2016} and with parametric importance sampling in \cite{owen+z:2000}, might allow to design even more efficient algorithms.

\paragraph{Theoretical perspectives.}
Non-asymptotic bounds would offer a different type of guarantee than the one provided in the paper: for a pre-specified probability level, one would have an error bound depending on $n$ and $\sigma_m$. In addition, the present work only considers the integration error for a single integrand whereas uniform bounds over some classes of integrands would be appropriate. Such results would apply to situations where many integrals are to be computed as for instance in likelihood-based inference for parametric models with latent variables. 

\paragraph{Regularization.}
As illustrated by the \textit{leverage condition}, the number of control variates at use needs to be limited but, in the mean time, the bound obtained, $n^{-1/2}\sigma_m$, is decreasing in the number of control variates. This advocates for selecting the most informative control variates before using them in the Monte Carlo procedure. Such an approach, based on the Lasso, has already been proposed in \cite{south+o+m+d:2018} and most recently, a pre-selection of the control variates, still by the Lasso, has been studied in \cite{leluc+p+s:2018}. The theoretical bounds obtained and the numerical illustration therein clearly advocate for pre-selecting the most effective control variates.

\paragraph{Un-normalized densities.}
Applications to Bayesian inference on models defined by un-normalized densities are not included in the present study. Two strategies might be conducted to handle such a situation. The first one consists in a normalized importance sampling approach. Suppose $h = (h_1, \ldots, h_m)'$ is a vector of control variates with respect to Lebesgue measure $\lambda$. Let $p$ denote the un-normalized target density and $q$ the importance sampling density. Let $(X_1,\ldots, X_n)$ be an independent random sample from $q$. Let $\hat \mu_n^{\mathrm{wOLS}}(f)$ denote the weighted OLS estimate defined as in \eqref{eq:OLS} but replacing $f$ by $fp/q$ and $h$ by $h/q$. Note that $\hat \mu_n^{\mathrm{wOLS}}(f)$ is an unbiased estimate of $\int fp \, \diff \lambda$. Because $p$ is un-normalized, the estimate cannot be computed and instead one needs to rely on the normalized version $\hat \mu_n^{\mathrm{wOLS}}(f) / \hat \mu_n^{\mathrm{wOLS}}(1)$. The second strategy follows from \cite{oates+c+b+g:2017} and relies on a Markov chain Monte Carlo approach. The control variates are defined through the Stein identity, see Eq.~(1) in the aforementioned paper. The sequence of integration points $(X_1, \ldots, X_n)$ is generated using the Metropolis--Hastings algorithm with target $p$. These two modifications allow to work with un-normalized densities. Non-trivial modifications of our proofs would be needed to analyse such procedures.

\section{Proofs}
\label{sec:proofs}

\begin{proof}[Proof of \eqref{eq:sigma:handy}]
	Put $\eps^{(n)} = (\eps(X_1), \ldots, \eps_n(X_n))'$. We have $f^{(n)} = \mu 1_n + \beta_{\mathrm{opt}}' H^{(n)} + \eps^{(n)}$. 
	Since $I_n - \Pi_{n,m}$ is the projection matrix on the orthocomplement in $\reals^n$ of the column space of $H^{(n)}$, we have by \eqref{eq:sigma:proj} that
	\begin{align*}
	\hat{\sigma}_{n,\mathrm{OLS}}^2
	&=
	\frac{1}{n}
	(\mu 1_n + \eps^{(n)} - \hat{\mu}_n^{\mathrm{OLS}} 1_n)'
	(I_n - \Pi_{n,m}) 
	(\mu 1_n + \eps^{(n)} - \hat{\mu}_n^{\mathrm{OLS}} 1_n) \\
	&=
	\frac{1}{n} (\eps^{(n)})' (I_n - \Pi_{n,m}) \, \eps^{(n)}  
	- \frac{1}{n} (\hat{\mu}_n^{\mathrm{OLS}} - \mu) 1_n' (I_n - \Pi_{n,m}) \{ 2 \eps^{(n)} - (\hat{\mu}_n^{\mathrm{OLS}} - \mu) 1_n \}.
	\end{align*}
	Replace $\Pi_{n,m}$ by the right-hand side in \eqref{eq:Pinm} to find
	\begin{multline*}
	\hat{\sigma}_{n,\mathrm{OLS}}^2
	=
	P_n(\eps^2) - P_n( \eps h' ) \, P_n(h h')^{-1} \, P_n(h \eps) \\
	- 
	(\hat{\mu}_n^{\mathrm{OLS}} - \mu)
	\{2 P_n(\eps) - 2 P_n(h') P_n(h h')^{-1} P_n(h \eps)\} \\
	+ (\hat{\mu}_n^{\mathrm{OLS}} - \mu)^2 \{1 - P_n(h') P_n(hh')^{-1} P_n(h)\}.
	\end{multline*}
	Equation~\eqref{eq:mu_OLS:2} and the identity $f = \mu + \beta_{\mathrm{opt}}' h + \eps$ imply that
	\[
	\hat{\mu}_n^{\mathrm{OLS}} - \mu 
	=
	\frac{P_n(\eps) - P_n(h') \, P_n(hh')^{-1} P_n(h\eps)}{1 - P_n(h') P_n(hh')^{-1} P_n(h)}.
	\]
	Use this identity to simplify the expression for $\hat{\sigma}_{n,\mathrm{OLS}}^2$ and arrive at \eqref{eq:sigma:handy}.
\end{proof}

The Euclidean norm of a vector $v$ is denoted by $|v| = (v'v)^{1/2}$. The corresponding matrix norm is $|A|_2 = \sup \{ |Av|/|v| : v \ne 0 \}$. The Frobenius norm of a rectangular matrix $A$ is given by $\lvert A \rvert_F = (\sum_i\sum_j A_{ij}^2)^{1/2} = \{\trace( A'A )\}^{1/2}$, with $\trace$ the trace operator. We have $|A|_2 \le |A|_F$, since $|A|_2^2$ is equal to the largest eigenvalue of $A'A$, while $|A|_F^2$ is equal to the sum of all eigenvalues of $A'A$, all of which are nonnegative. Recall the cyclic property of the trace operator: for matrices $A$ and $B$ of dimensions $k \times \ell$ and $\ell \times k$, respectively, we have $\trace(AB) = \trace(BA)$.

Recall that the Gram matrix $P(h_n h_n')$ was assumed to be invertible. Let $I_k$ denote the $k \times k$ identity matrix. Let $B_n$ be an $m_n \times m_n$ matrix such that $B_n' B_n = P(h_n h_n')^{-1}$; use for instance the eigendecomposition of $P(h_n h_n')$ to construct $B_n$. Clearly, $B_n$ is invertible. The OLS estimator based on the transformed vector of control functions
\[ 
\hbar_n = (\hbar_{n,1}, \ldots, \hbar_{n,m_n})' = B_n h_n 
\] 
is therefore identical to the one based on $h_n$. The transformed vector $\hbar_n$ has the advantage that its elements are orthonormal, i.e., its Gram matrix is equal to the identity matrix:
\begin{equation}
\label{eq:Id}
P(\hbar_n \hbar_n') = B_n \, P(h_n h_n') \, B_n' = B_n (B_n' B_n)^{-1} B_n' = I_{m_n}.
\end{equation}
The function $q_n$ defined in \eqref{eq:q} is equal to $q_n = \hbar_n' \hbar_n$.

\begin{lemma}
	We have
	\begin{align}
	\label{eq:Pnhn}
	\expec\{ |P_n(h_n)|^2 \}
	&=
	n^{-1} P(h_n' h_n), \\
	\label{eq:Pnhn:ortho}
	\expec\{ |P_n(\hbar_n)|^2 \}
	&=
	m_n / n.
	\end{align}
\end{lemma}

\begin{proof}
	We have
	\[
	|P_n(h_n)|^2
	=
	\frac{1}{n^2} \sum_{i=1}^n \sum_{j=1}^n h_n'(X_i) \, h_n(X_j).
	\]
	The random variables $X_1, \ldots, X_n$ form an independent random sample from $P$. Furthermore, $P(h_n) = 0$. As a consequence, 
	\[
	\expec\{ |P_n(h_n)|^2 \} = n^{-1} \expec\{ h_n'(X_1) \, h_n(X_1) \} = n^{-1} P(h_n' h_n),
	\]
	yielding \eqref{eq:Pnhn}. Equation~\eqref{eq:Pnhn:ortho} follows from \eqref{eq:Pnhn} and $P(\hbar_n' \hbar_n) = P(q_n) = m_n$, see \eqref{eq:Pq}.
\end{proof}

\begin{lemma}
	\begin{equation}
	\label{eq:Gram:Frobenius}
	\expec\{ |P_n(\hbar_n \hbar_n') - I_{m_n}|_F^2 \}
	=  n^{-1} \{ P_n(q_n^2) - m_n \}.
	\end{equation}
\end{lemma}

\begin{proof}
	We have $P_n(\hbar_n \hbar_n') - I_{m_n} = n^{-1} \sum_{i=1}^n A_{n,i}$ with $A_{n,i} = \hbar_n(X_i) \hbar_n'(X_i) - I_{m_n}$. Since the matrix $\hbar_n \hbar_n'$ is symmetric and since the trace operator is linear,
	\begin{align*}
	\expec\{ |P_n(\hbar_n \hbar_n') - I_{m_n}|_F^2 \}
	&=
	\expec(\trace[\{P_n(\hbar_n \hbar_n') - I_{m_n}\}^2]) \\
	&=
	\trace(\expec[\{P_n(\hbar_n \hbar_n') - I_{m_n}\}^2]) \\
	&=
	\frac{1}{n^2} \sum_{i=1}^n \sum_{j=1}^n
	\trace\{\expec(A_{n,i} A_{n,j})\}.
	\end{align*}
	The triangular array of random matrices $(A_{n,i})_{n,i}$ is rowwise iid; the random matrices $A_{n,i}$ are square integrable and centered. If $i \ne j$, then $\expec[A_{n,i} A_{n,j}] = 0$, the $m_n \times m_n$ null matrix. Hence
	\[
	\expec\{ |P_n(\hbar_n \hbar_n') - I_{m_n}|_F^2 \} 
	= n^{-1} \trace\{ \expec( A_{n,1}^2 ) \}.
	\]
	By the cyclic property of the trace,
	\[
	\trace\{ \expec( A_{n,1}^2 ) \}
	=
	\trace[ P\{(\hbar_n \hbar_n')^2\} - I_{m_n} ] \\
	=
	P\{ (\hbar_n' \hbar_n)^2 \} - m_n.
	\]
	Since $\hbar_n' \hbar_n = q_n$, the equality \eqref{eq:Gram:Frobenius} follows.
\end{proof}

\begin{lemma}
	\begin{equation}
	\label{eq:Gram:invertible}
	\pr\{ \text{$P_n(h_n h_n')$ is not invertible} \}
	\le
	n^{-1} P (q_n^2)
	\end{equation}
\end{lemma}

\begin{proof}
	Since $\hbar = B_n h_n$ and since $B_n$ is invertible, the matrix $P_n(h_n h_n')$ is invertible if and only if the matrix $P_n(\hbar_n \hbar_n')$ is so. Suppose $P_n(\hbar_n \hbar_n')$ is not invertible. Then there exists a nonzero vector $v \in \mathbb{R}^{m_n}$ such that $P_n(\hbar_n \hbar_n') v = 0$ and thus $\{ P_n(\hbar_n \hbar_n') - I_{m_n} \} v = -v$. It then follows that
	\[
	|P_n(\hbar_n \hbar_n') - I_{m_n}|_F \ge |P_n(\hbar_n \hbar_n') - I_{m_n}|_2 \ge 1.
	\]
	But since $I_{m_n} = P(\hbar_n \hbar_n')$ by \eqref{eq:Id}, equation~\eqref{eq:Gram:Frobenius} yields
	\begin{align*}
	\pr\{ \text{$P_n(h_n h_n')$ is not invertible} \}
	&\le 
	\pr\{ |P_n(\hbar_n \hbar_n') - I_{m_n}|_F \ge 1 \} \\
	&\le
	\expec\{ |P_n(\hbar_n \hbar_n') - I_{m_n}|_F^2 \}
	\le
	n^{-1} P(|\hbar_n|^4).
	\end{align*}
	Finally, $|\hbar_n|^4 = (\hbar_n' \hbar_n)^2 = q_n^2$.
\end{proof}

\begin{lemma}
	\label{lem:cond1:consequences}
	If Condition~\ref{cond:q} holds, then $P_n(h_n h_n')$ and $P_n(\hbar_n \hbar_n')$ are invertible with probability tending to one as $n \to \infty$ and
	\begin{align} 
	\label{eq:Graminverse:ortho}
	\lvert P_n(\hbar_n \hbar_n')^{-1} \rvert_2 
	&\le 1 + \oh_{\pr}(1), \\
	\label{eq:denominatorbound}
	P_n(h_n') \, P_n(h_n h_n')^{-1} \, P_n(h_n)
	&=
	\Oh_{\pr}(m_n / n).
	\end{align}
\end{lemma}

\begin{proof}
	In view of~\eqref{eq:Gram:invertible}, the first part of Condition~\ref{cond:q} implies that $P_n(h_n h_n')$ and thus $P_n(\hbar_n \hbar_n')$ are invertible with probability tending to one. 
	
	Write $J_n = P_n(\hbar_n \hbar_n')$. On the event that $P_n(h_n h_n')$ is invertible, $J_n$ is invertible too, and $J_n^{-1} = I_{m_n} + J_n^{-1} (I_{m_n} - J_n)$ and thus $|J_n^{-1}|_2 \le 1 + |J_n^{-1}|_2 \, |I_{m_n} - J_n|_2$ by multiplicativity of the matrix norm $| \, \cdot \, |_2$. It follows that, provided $|I_{m_n} - J_n|_2 < 1$, we have
	\[
	|J_n^{-1}|_2 \le \frac{1}{1 - |I_{m_n} - J_n|_2}.
	\]
	Recall that $B_n' B_n = P(h_n h_n')^{-1}$. By an application of \eqref{eq:Gram:Frobenius} to the orthonormalized functions $\hbar_n = B_n h_n$, we have
	\[
	\expec( | J_n - I_{m_n} |_F^2 )
	\le
	n^{-1} P(|\hbar_n|^4)
	=
	n^{-1} P(q_n^2)
	=
	\oh(1)
	\]
	as $n \to \infty$, in view of \eqref{eq:on}. Therefore, $|I_{m_n} - J_n|_2 \le |I_{m_n} - J_n|_F = \oh_{\pr}(1)$. We conclude that $|J_n^{-1}|_2 \le 1 + \oh_{\pr}(1)$.
	
	Secondly, since
	\[
	P_n(h_n') \, P_n(h_n h_n')^{-1} \, P_n(h_n)
	=
	P_n(\hbar_n') \, P_n(\hbar_n \hbar_n')^{-1} \, P_n(\hbar_n),
	\]
	we have
	\[
	\lvert P_n(h_n') \, P_n(h_n h_n')^{-1} \, P_n(h_n) \rvert
	\le
	\lvert P_n(\hbar_n) \rvert^2 \, |J_n^{-1}|_2.
	\]
	We have just shown that $|J_n^{-1}|_2 = \Oh_{\pr}(1)$. Furthermore, $\lvert P_n(\hbar_n) \rvert^2 = \Oh_{\pr}(m_n/n)$ by \eqref{eq:Pnhn:ortho} and Markov's inequality.
\end{proof}

Recall that $f_n = g_n + \eps_n$, where $g_n$ is the orthogonal projection of $f_n$ on the linear subspace of $L^2(P)$ spanned by $\{1, h_{n,1}, \ldots, h_{n,m_n}\}$.

\begin{lemma}
	We have
	\begin{equation}
	\label{eq:Pheps}
	\expec \{ |P_n( h_n \eps_n )|^2 \}
	=
	n^{-1} P (|h_n|^2 \eps_n^2).
	\end{equation}
	If Condition~\ref{cond:q} holds, we have therefore
	\begin{equation}
	\label{eq:Pheps:ortho}
	\lvert P_n( \hbar_n \eps_n ) \rvert
	=
	\oh_{\pr} ( m_n^{-1/2} \sigma_n ), \qquad n \to \infty.
	\end{equation}
\end{lemma}

\begin{proof}
	We have
	\[
	|P_n( h_n \eps_n )|^2
	= \frac{1}{n^2} \sum_{i=1}^n \sum_{j=1}^n
	h_n'(X_i) \, h_n(X_j) \, \eps_n(X_i) \, \eps_n(X_j).
	\]
	Since $P( h_{n,k} \eps_n ) = 0$ for all $k = 1, \ldots, m_n$ and since the variables $X_1, \ldots, X_n$ are iid $P$, we have $\expec\{ |P_n( h_n \eps_n )|^2 \} = n^{-1} \expec\{ h_n'(X_1) h_n(X_1) \, \eps_n(X_1)^2 \}$, yielding \eqref{eq:Pheps}.
	
	Apply \eqref{eq:Pheps} to $\hbar_n$; since $|\hbar_n|^2 = \hbar_n' \hbar_n = q_n$, we find
	\[
	\expec \{ |P_n( \hbar_n \eps_n )|^2 \}
	=
	n^{-1} P( |\hbar_n|^2 \eps_n^2 )
	=
	n^{-1} P( q_n \eps_n^2 )
	=
	\oh( m_n^{-1} \sigma_n^2 )
	\]
	as $n \to \infty$, by \eqref{eq:onm}.
\end{proof}

\begin{proof}[Proof of Theorem~\ref{thm:rate}]
	On an event $E_n$ with probability tending to one, $P_n(h_n h_n')$ is invertible and $P_n(h_n') \, P_n(h_n h_n)^{-1} \, P_n(h_n)$ is less than $1$ (Lemma~\ref{lem:cond1:consequences}). On $E_n$, the OLS estimator is given by \eqref{eq:mu_OLS:2}. Substitute $f_n = \mu_n + \beta_n' h_n + \eps_n$ to see that, on $E_n$, we have
	\[
	\sqrt{n}(\hat{\mu}_n^{\mathrm{OLS}} - \mu_n)
	=
	\sqrt{n}
	\frac%
	{P_n(\eps_n) - P_n(\eps_n h_n') \, P_n(h_n h_n')^{-1} P_n(h_n)}%
	{1 - P_n(h_n') \, P_n(h_n h_n')^{-1} \, P_n(h_n)}.
	\]
	By \eqref{eq:denominatorbound}, the denominator is $1 + \oh_{\pr}(1)$ as $n \to \infty$. The second term in the numerator does not change if we replace $h_n$ by $\hbar_n$. Its absolute value is bounded by
	\[
	|P_n(\hbar_n \eps_n)| \, |P_n(\hbar_n \hbar_n')^{-1}|_2 \, |P_n(\hbar_n)|
	=
	\oh_{\pr}(m_n^{-1/2} \sigma_n) \, \Oh_{\pr}(1) \, \Oh_{\pr}\{(m_n/n)^{1/2}\}
	=
	\oh_{\pr}(n^{-1/2} \sigma_n);
	\]
	here we used \eqref{eq:Pheps:ortho}, \eqref{eq:Graminverse:ortho}, and \eqref{eq:Pnhn:ortho}, respectively. We find
	\[
	\sqrt{n}(\hat{\mu}_n^{\mathrm{OLS}} - \mu_n)
	=
	\sqrt{n} \{1 + \oh_{\pr}(1)\} P_n(\eps_n) + \oh_{\pr}(\sigma_n).
	\]
	Since $\expec\{ P_n(\eps_n)^2 \} = n^{-1} \sigma_n^2$, we have $P_n(\eps_n) = \Oh_{\pr}( n^{-1/2} \sigma_n )$. We conclude that
	\[
	\sqrt{n}(\hat{\mu}_n^{\mathrm{OLS}} - \mu_n)
	= \sqrt{n} \, P_n(\eps_n) + \oh_{\pr}(\sigma_n).
	\]
	Divide both sides by $\sigma_n$ to conclude the proof of Theorem~\ref{thm:rate}.
\end{proof}

\begin{proof}[Proof of Theorem~\ref{thm:AN}]
	By the Lindeberg--Feller central limit theorem \citep[Theorem~5.12]{kallenberg:2002} applied to the triangular array $\{ \eps_n(X_i) : i = 1, \ldots, n\}$ of rowwise iid random variables, Condition~\ref{cond:Lindeberg} is necessary and sufficient for $(\sqrt{n}/\sigma_n) P_n(\eps_n)$ to be asymptotically standard normal. In view of \eqref{eq:AN:prob} and Slutsky's lemma, $(\sqrt{n}/\sigma_n) P_n(\eps_n)$ is asymptotically standard normal if and only if $(\sqrt{n}/\sigma_n) ( \hat{\mu}_n^{\mathrm{OLS}} - \mu )$ is asymptotically standard normal.
	
	We prove \eqref{eq:sigma:consistent}. As in the proof of Theorem~\ref{thm:rate}, there is a sequence $E_n$ of events with probability tending to one such that on $E_n$, the matrix $P_n(h_n h_n')$ is invertible and such that $P_n(h_n') \, P_n(h_n h_n')^{-1} \, P_n(h_n) < 1$. On $E_n$, the OLS estimator of $\sigma_n^2$ is given by \eqref{eq:sigma:handy}. Clearly, we can replace $h_n$ by $\hbar_n = B_n h_n$ and find
	\begin{multline*}
	\hat{\sigma}_{n,\mathrm{OLS}}^2
	=
	P_n(\eps_n^2) - P_n( \eps_n \hbar_n' ) \, P_n(\hbar_n \hbar_n')^{-1} \, P_n(\hbar_n \eps_n) \\
	- 
	(\hat{\mu}_n^{\mathrm{OLS}} - \mu)^2 \{1 - P_n(\hbar_n') P_n(\hbar_n \hbar_n')^{-1} P_n(\hbar_n)\}.
	\end{multline*}
	The bounds established in the course of the proof of Theorem~\ref{thm:rate} together with the fact that $(\hat{\mu}_n^{\mathrm{OLS}} - \mu)^2 = \Oh_{\pr}(n^{-1} \sigma_n^2)$ easily yield
	\[
	\hat{\sigma}_{n,\mathrm{OLS}}^2
	=
	P_n(\eps_n^2) + \oh_{\pr}(m_n^{-1} \sigma_n^2).
	\]
	It then suffices to show that $P_n(\eps_n^2) / \sigma_n^2 = 1 + \oh_{\pr}(1)$. But this is a consequence of Proposition~\ref{prop:LLN} below applied to the triangular array $Y_{n,i} = \eps_n^2(X_i) / \sigma_n^2$. The Lindeberg condition is exactly condition~\eqref{eq:LLN:cond} in that Proposition.
\end{proof}

\begin{proposition}
	\label{prop:LLN}
	Let $\{ Y_{n,i} : 1 \le i \le n \}$ be a triangular array of nonnegative, rowwise iid random variables with unit expectation. If, as $n \to \infty$, for all $\delta > 0$, we have
	\begin{equation}
	\label{eq:LLN:cond}
	\expec[ Y_{n,1} \, \1\{ Y_{n,1} > \delta n \} ] = \oh(1),
	\end{equation}
	then $n^{-1} \sum_{i=1}^n Y_{n,i} = 1 + \oh_\pr(1)$.
\end{proposition}

\begin{proof}
	We apply \citep[Theorem~2.2.6]{durrett:2010} with $a_n = b_n = n$. We need to check two conditions: (i) $n \pr( Y_{n,1} > n ) \to 0$ and (ii) $n^{-1} \expec[ Y_{n,1}^2 \1\{ Y_{n,1} \le n \} ] \to 0$ as $n \to \infty$.
	
	Condition~(i) follows at once from $n \pr(Y_{n,1} > n) \le \expec[ Y_{n,1} \1\{ Y_{n,1} > n \}]$ and \eqref{eq:LLN:cond}.
	
	Regarding condition~(ii), choose $\delta \in (0, 1]$ and note that, since $\expec[Y_{n,1}] = 1$, we have
	\begin{align*}
	n^{-1} \expec[ Y_{n,1}^2 \1\{ Y_{n,1} \le n \} ]
	&=
	n^{-1} \expec[ Y_{n,1}^2 \1\{ Y_{n,1} \le \delta n \} ]
	+
	n^{-1} \expec[ Y_{n,1}^2 \1\{ \delta n < Y_{n,1} \le n \} ] \\
	&\le
	\delta + \expec[ Y_{n,1} \1\{ \delta n < Y_{n,1} \}].
	\end{align*}
	The $\limsup$ as $n \to \infty$ is bounded by $\delta$ because of \eqref{eq:LLN:cond}. Since $\delta$ was arbitrary, condition~(ii) follows.
\end{proof}

\acks

The authors are grateful to Chris Oates and to two anonymous reviewers for useful comments and additional references. The authors gratefully acknowledge support from the Fonds de la Recherche Scientifique (FNRS) A4/5 FC 2779/2014-2017 No.\ 22342320, from the contract ``Projet d'Act\-ions de Re\-cher\-che Concert\'ees'' No.\ 12/17-045 of the ``Communaut\'e fran\c{c}aise de Belgique'' and from the IAP research network Grant P7/06 of the Belgian government (Belgian Science Policy).

\bibliographystyle{apt}
\bibliography{biblio}

\end{document}